\documentclass[a4paper,11pt]{article}
\usepackage[makeroom]{cancel}
\usepackage[table]{xcolor}
\usepackage{amsmath}
\usepackage{amsfonts}
\usepackage{geometry}
\usepackage[utf8]{inputenc}
\usepackage{enumerate} 
\usepackage{subfig}
\usepackage{tikz}
\usepackage{color}
\usepackage{float}
\usepackage{leftidx}
\usepackage{bigints}
\usepackage[mathscr]{euscript}
\usepackage{amssymb,amsmath,dsfont,url,epsfig}
\usepackage{titlesec, blindtext, color}
\usepackage{graphicx,psfrag,epsf,times}
\definecolor{gray75}{gray}{0.75}

\newcommand{\sln}{\linespread{1}}
\newcommand*{\email}[1]{\href{mailto:#1}{\nolinkurl{#1}} } 
\titleformat{\chapter}[block]{\LARGE\bfseries\sln}{Chapter \thechapter}{11pt}{\newline\huge\bfseries}
\newtheorem{thm}{Theorem}[section]
\newtheorem{rem}{Remark}[section]
\newtheorem{defn}{Definition}[section]

\newenvironment{proof}{\paragraph{Proof:}}{\hfill$\square$}
\newtheorem{lem}{Lemma}[section]
\newtheorem{proposition}{Proposition}[section]
\newtheorem{corollary}{Corollary}[section]
\begin{document}
\title{The Funk-Finsler structure on the unit disc in the hyperbolic plane}
\author{Ashok Kumar\footnote{E-mail: ashok241001@bhu.ac.in ; Centre for Interdisciplinary Mathematical Sciences, Banaras Hindu University, Varanasi-221005, India}, Hemangi Madhusudan Shah\footnote{E-mail: hemangimshah@hri.res.in ; Harish-Chandra Research Institute, A CI of Homi Bhabha National Institute, Chhatnag Road, Jhunsi, Prayagraj-211019, India.} and Bankteshwar Tiwari\footnote{E-mail: btiwari@bhu.ac.in; Centre for Interdisciplinary Mathematical Sciences, Banaras Hindu University, Varanasi-221005, India}}
   \maketitle

\begin{abstract}
\noindent 
 In this paper, we {\it construct}  Funk-Finsler structure in various models of the hyperbolic plane. In particular,  in the unit disc of the Klein model, it turns out to be a Randers metric, which is a non- Berwald Douglas metric. Further, using Finsler isometries we obtain the  Funk-Finsler structures in other models of 
 the hyperbolic plane. 
Finally, we also investigate the geometry of this Funk-Finsler metric by explicitly computing the  $S$-curvature, Riemann curvature, flag curvature, and the Ricci curvature in the Klein unit disc. 
\end{abstract}
\section{Introduction}
The Probe into the Funk and the Hilbert metrics was begun in order to solve the Hilbert's fourth problem: {\it Find all the non-Euclidean metrics having the line segments as their geodesics}. The Funk metric on the unit ball in the Euclidean space is a well-known Finsler metric of constant flag curvature $-\frac{1}{4}$; whereas the Hilbert metric on the unit ball, which is the arithmetic symmetrization of its Funk metric, is of constant curvature $-1$. The Funk metric on the unit disc is also an example of Randers metric on the unit ball, which can be realized as the deformation of the well-known Klein metric on the ball by a closed $ 1$-form. \\
The Funk distance  $d_{F, \Omega}(x, y)$ between any two points $x$ and $y$ in a convex set $\Omega$ in the Euclidean space $\mathbb{R}^n$ is given by
\begin{equation*}
d_{F, \Omega}(x, y):=
  \left\{
  \begin{array}{ll}
 \log\left( \frac{|x-a|}{|y-a|} \right), & \mbox{if }  x \neq y \\
0, & \mbox{if } x = y. 
\end{array}
 \right.
\end{equation*} 
where $|.|$ represents the Euclidean norm in $\mathbb{R}^n$, $\mathfrak{a}=\overrightarrow{xy} \cap \partial \Omega$ and $\overrightarrow{xy}$  denotes the  ray starting from $x$ and passing through $y$  and  $\partial \Omega=\bar{\Omega} \setminus \Omega$ denotes the  boundary of $\Omega$ (\cite{HHG}, Chapter 2, $\S 2$).\\ 
Also, the corresponding  Finsler structure on the convex set $\Omega$  is given by
    \begin{equation*}
   F_\Omega(x, \xi) := \inf  \left\lbrace t > 0 ~\vert  \left( x+\frac{\xi}{t} \right) \in   \Omega \right\rbrace,
    \end{equation*}
    where $\xi \in T_x\Omega$ ( \cite{HHG}, Chapter 3, $\S 3$ ).

    \vspace{.3cm}
    \noindent
Recently, Papadopoulos and Yamada (\cite{PAYS}) defined the Finsler structure for the Funk and the Hilbert metric on a convex set in the constant curvature spaces. In particular, the Funk distance  $d_{F, \Omega}(x, y)$ between any two points $x$ and $y$ in a convex set $\Omega$ in the upper half plane is given by,
 \begin{equation*}
d_{F,\Omega}(x, y):=\log\frac{\sinh d_U(x,\mathfrak{a})}{\sinh d_U(y,\mathfrak{a})},
\end{equation*}
where  $d_U$ represents the hyperbolic distance in upper half plane, $\mathfrak{a}=\overrightarrow{xy} \cap \partial \Omega$ and $\overrightarrow{xy}$ denotes the hyperbolic ray starting from $x$ and passing through $y$ (see \cite{PAYS},$\S 3$ ). 
 The corresponding  Finsler structure in the convex set $\Omega$ in the upper half plane is given by,
\begin{equation*}
    \mathcal{F}(x, \xi)=\sup \limits_{\pi \in \mathscr{P} } \frac{\cosh d_U(x, T(x,\xi, \pi))}{\sinh d_U(x, T(x,\xi,\pi))} ||\xi||_U.
\end{equation*}
 Here $x \in  \Omega$,  $\xi \in  T_x\Omega$, $||.||_U$ represents the hyperbolic norm in upper half plane, $\pi$ is a supporting hyperplane on  $\partial \Omega$, $\mathscr{P}$ is the collection of all the supporting hyperplanes on  $\partial \Omega$ and $T(x, \xi, \pi)=R_{x,\xi} \cap \pi$, where $R_{x,\xi}$ denotes  the geodesic ray starting from $x$ with initial tangent vector $\xi$ (see \cite{PAYS},$\S 3$ ).\\\\
Inspired by the above idea of Papadopoulos et al., we {\it construct} the Funk-Finsler metric on the Klein unit disc in the  Klein hyperbolic disc model;
we term this as the {\it Funk-Finsler metric on the Klein unit disc} (see \eqref{eqn 4.31}). We also show that this metric indeed is a Randers metric with a closed $1$-form, and therefore it is a Douglas  metric (see \cite{CXSZ}, Theorem $5.2.1$) but \textit{not} Berwald (see \cite{SSZ}, Lemma $3.1.2$).\\
We explicitly, show that this Funk-Finsler metric on the Klein unit disc $\mathbb{D}_K(1)$ can be realized as the pull back of a Lorentz-Randers metric $F_L=\alpha_L + \beta_L$  (see \eqref{eqn3.1222}) on the upper sheet of the hyperboloid of two sheets. \\
We further show that this Funk-Finsler metric on the  Klein unit disc $\mathbb{D}_K(1)$ can also be realized as the pull back of a  Randers metric $F_+ = \alpha_+ + \beta_+$ (see \eqref{eqn 3.5})   on the upper hemisphere of radius $r$. Thus, interestingly we see that the Funk-Finsler metric on the unit disc is obtained via two different types of metrics.\\
\noindent
Recently, Mester and Kristály (\cite{AMAK}) have investigated the equivalence among three models of Finsler-Poincar\'e disc and the present authors (see \cite{AHB}) have studied the realization of these models via pullbacks of certain Randers metric on upper sheet of hyperboloid and on upper hemi sphere and also found two other equivalent models.\\ 
In this paper, we also build the two equivalent models of the Funk-Finsler structure on the Klein unit disc. 
The first one from the Poincar\'e disc model and the other one from the Poincar\'e upper-half model. Finally, we 
further probe into the geometry of this
Funk-Finsler metric in the Klein unit disc and 
obtain some important curvatures like $S$, Ricci curvature, and flag curvatures. 
%We also find the geodesic
%equation of this metric and observe that straight lines in the unit disc must be geodesics. \\

\noindent
The paper is organised as follows.\\
In Section $2$, we discuss the preliminaries required for the paper.
In Section $3$,  we construct explicitly the Funk-Finsler metric on the Klein unit disc in the Klein hyperbolic disc model $\mathcal{D}_{K}$. And we also find the aforementioned two equivalent models.

% Also, we show the Funk structure on the Klein unit disc is %the pullback of the deformed  Lorentzian metric on the upper %half of the hyperboloid of two sheets(see 3.1). It is %another pullback of the deformation of the hyperbolic metric %on $\mathbb{R}^3_+$  of the upper hemisphere with radius $r$ %(see 3.2).\\
%We introduce and explore the $2$ isometric models of the %Funk disc in detail. The first is equivalent to the %Poincar\'e unit disc in the Poincar\'e disc model (see 3.3). %The second is equivalent to the unit disc in the Poincar\'e %upper-half model (see 3.4).\\

\noindent
In Section $4$,  we investigate geometry of the Funk-Finsler
metric on the Klein unit disc and compute explicitly  $\textbf{S}$-curvature, spray coefficients, Riemann curvature,  Ricci curvature, and the flag curvature. Finally, we also obtain the Zermelo Navigation data for the Funk-Finsler metric on the Klein unit disc, which will be useful for furthermore investigations of this metric in future.\\

In the sequel, the following notations will be used throughout the paper. \\

\textbf{Notations:}
\begin{itemize}
\item $|.|$ and $\langle . \rangle$, respectively, represent 
the Euclidean norm and the Euclidean inner product, respectively.
    % \item $\mathbb{D} = \left\lbrace(x^1, x^2)\in \mathbb{R}^2 : (x^1)^2+(x^2)^2 < 1\right\rbrace,$ the unit disc in $\mathbb{R}^2$.
     \item $\mathbb{D}_E( \mathfrak{r})=\left\lbrace (x^1,x^2) \in \mathbb{R}^2 :  (x^1)^2+(x^2)^2 <  \mathfrak{r}^2 \right\rbrace$, the Euclidean disc centred at origin with radius $ \mathfrak{r}$ in $\mathbb{R}^2$.     
\item $\mathbb{U}=\left\lbrace(x^1, x^2)\in \mathbb{R}^2 : x^2 > 0\right\rbrace,$  the {upper half plane}.
\item $\mathbb{H_+} = \left\lbrace(\tilde{x}^1, \tilde{x}^2, \tilde{x}^3)\in \mathbb{R}^3 : \tilde{x}^3 = \sqrt{1+(\tilde{x}^1)^2+(\tilde{x}^2)^2} \right\rbrace,$ the upper half of the hyperboloid of two sheets.
\item $\mathbb{S}^2_+( \mathfrak{r})=\{(x^1,x^2,x^3)\in \mathbb{R}^3 :  (x^1)^2 +(x^2)^2+(x^3)^2= \mathfrak{r}^2 \ \text{and} \ x^3 > 0\},$ the
upper half of the hemisphere with radius $ \mathfrak{r}$ in $\mathbb{R}^3$.
     \item $\mathcal{D}_{K}$, the Euclidean unit disc in $\mathbb{R}^2$ centred at origin equipped with the Klein hyperbolic metric.
       \item  $\mathbb{D}_K(1)$, the Klein unit disc in $\mathcal{D}_{K}$ centred at origin.
     \item $\mathcal{D}_{P}$,  the Euclidean unit disc in $\mathbb{R}^2$ centred at origin equipped with the Poincar\'e hyperbolic  metric.
     \item  $\mathbb{D}_P(1)$, the Poincar\'e unit disc in $\mathcal{D}_{P}$ centred at origin.
     \item   $\mathbb{D}_U(a, \mathfrak{r})$, the  hyperbolic disc in $\mathbb{U}$ with centered at $a$ and hyperbolic radius $\mathfrak{r}$.
\end{itemize}

\section{Preliminaries}
The theory of Finsler manifolds can be considered as a generalization of  Riemannian manifolds, where the Riemannian metric is replaced by a so-called {\it Finsler} metric. A Finsler metric is a smoothly varying family of Minkowski norms in each tangent space of the manifold.\\
Let $ M $ be an $n$-dimensional smooth manifold and let $T_{x}M$ denote the tangent space of $M$ at $x$. The tangent bundle $TM$  of $M$ is the disjoint union of tangent spaces: $TM:= \sqcup _{x \in M}T_xM $. We denote the elements of $TM$ by $(x,\xi)$, where $\xi \in T_{x}M $ and $TM_0:=TM \setminus\left\lbrace 0\right\rbrace $.
 \begin{defn}[Finsler structure \cite{SSZ}, \boldmath$\S 1.2$]\label{def 3.A01}
   A {\it Finsler structure} on the manifold $M$ is a function $F:TM \to [0,\infty)$ satisfying the following conditions: 
   \begin{enumerate}[(i)]
       \item  $F$ is smooth on $TM_{0}$,
      \item  $F$ is a positively 1-homogeneous on the fibers of the tangent bundle $TM$,
 \item  The Hessian of $\displaystyle \frac{F^2}{2}$ with elements $\displaystyle g_{ij}=\frac{1}{2}\frac{\partial ^2F^2}{\partial \xi^i \partial \xi^j}$ is positive definite on $TM_0$.\\
 The pair $(M, F)$ is called a Finsler space, and $g_{ij}$ is called the fundamental tensor of the Finsler structure $F$.
   \end{enumerate}
 
 \end{defn}
 
 \noindent
 It is easy to see that Riemannian metrics are examples of Finsler metrics. The Randers metric is the simplest non-Riemannian example of a Finsler metric.
\begin{defn}[Randers Metric \cite{SSZ}, \boldmath$\S 1.2$]  Let $\displaystyle \alpha=\sqrt{a_{ij}(x)dx^idx^j}$ be a Riemannian metric  and $\beta=b_i(x)dx^i$ be a one-form on a smooth manifold  $M$ with $||\beta||_{\alpha}<1$, where $\displaystyle ||\beta||_{\alpha} = \sqrt {a^{ij}(x)b_{i}(x)b_{j}(x)}$; \  then $F(x,\xi)=\alpha(x,\xi)+\beta(x,\xi)$, for all $x \in M, \xi \in T_xM$, is called a Randers metric on $M$. 
\end{defn}
 
\noindent In what follows, we will be dealing with the
{\it Funk} and the {\it Hilbert} metric  on a strongly convex domain $\Omega$ in $\mathbb{R}^n$.

%\vspace{.3cm}
%\noindent
%Let $\Omega$ be a non empty   convex domain in $\mathbb{R}^n$ and let  $\partial \Omega=\bar{\Omega} \setminus \Omega$ denote the  boundary of $\Omega$. For any two points $x^1, x^2; x^1 \neq x^2$ in  $\Omega$, $\overrightarrow{x^1x^2}$  denotes the  ray starting at $x^1$ and passing through $x^2$ and $\overleftarrow{x^1x^2}$ denotes the  ray starting at $x^2$ and passing through $x^1$.% In the sequel, $|.|$ denotes the usual Euclidean norm in $\mathbb{R}^n$. 
 \begin{defn}
   {(\bf{Funk metric} \cite{HHG}, Chapter 2, \boldmath$\S 2$)} The {\it Funk} metric on a  convex domain $\Omega \subset \mathbb{R}^n$ is denoted by $d_{F,\Omega}$ and is defined as:  
 \\ For any  $x$ and $y$ in $\Omega$ and  $m=\overrightarrow{xy}\cap \partial \Omega$,
 
\begin{equation}\label{s1}
d_{F,\Omega}(x,y)=
  \left\{
  \begin{array}{ll}
 \log\left( \frac{|x-m|}{|y-m|} \right), & \mbox{if }  x \neq y \\
0, & \mbox{if } x = y. 
\end{array}
 \right.
\end{equation} 
\end{defn}

 \begin{defn} [\textbf{Hilbert metric} \cite{HHG}, Chapter 3, \boldmath$\S 4$]
 The {\it Hilbert} metric on a  convex domain $\Omega \subset \mathbb{R}^n$,  denoted by $d_{H,\Omega}$,  is defined as: \\ 
 For any  $x$ and $y$ in $\Omega$, let $m=\overrightarrow{xy}\cap \partial \Omega$ and $\Bar{m}=\overleftarrow{xy}\cap \partial \Omega$, where $\overleftarrow{xy}$ denotes the  ray starting from $y$ and passing through $x$. Then
 \begin{equation}\label{s12}
d_{H,\Omega}(x,y)=
  \left\{
  \begin{array}{ll}
\frac{1}{2} \log \left( \frac{|y-\Bar{m}|.|x-m|}{|x-\Bar{m}|.|y-m|}\right), & \mbox{if }  x \neq y \\
0, & \mbox{if } x = y.
\end{array}
 \right.
\end{equation} 

 \end{defn}
 
\vspace{0.1in}
\noindent
 It can be easily shown that the function $d_{F,\Omega}$ 
 is an asymmetric distance function on $\Omega$,
 while the function $d_{H,\Omega}$  is the {\it legitimate} distance function on $\Omega$. These distances satisfy an interesting property that, along line segments between any two points, they are minimal. 

\begin{defn}[\textbf{Geodesics and Spray coefficients \cite{SSZ}, $\S 3.1$}] \label{defaa} 
   \noindent A smooth curve $\gamma$ in the Finsler manifold $(M^n, F)$ is a geodesic if and only if $\gamma(t)=(x^i(t))$ satisfies the differential equations:
   \begin{equation}\label{eqn2.1.13}
   \frac{d^2x^i(t)}{dt^2}+2{G}^i\left(\gamma,\frac{d \gamma}{dt}\right)=0, \qquad 1 \le i \le n.
 \end{equation}
   Here $G^i=G^i(x,\xi)$ are local functions on $TM$, called the spray coefficients defined by 
   \begin{equation}\label{eqn2.1.14}
  {G}^i=\frac{1}{4}{g}^{i{\ell}}\left\{\left[{F}^{2}\right]_{x^k\xi^{\ell}}\xi^k-	\left[{F}^{2}\right]_{x^{\ell}}\right\}.
  	 		 \end{equation} 
\end{defn}
    
 \begin{defn}  [\textbf{The Riemann curvature tensor} \cite{CXSZ}, $\S 4.1$] \label{def2}  
 	The Riemann curvature tensor $R={{R}_\xi:T_xM \rightarrow T_xM}$, for a Finsler space $(M^n,F)$ is defined by
 ${R}_{\xi}(u)={R}^{i}_{k}(x,\xi)u^{k} \frac{\partial}{\partial x^i}$,
 	\ $u=u^k\frac{\partial}{\partial x^k}$,\ where ${R}^{i}_{k}={R}^{i}_{k}(x,\xi)$ denote the coefficients of the Riemann curvature tensor of the Finsler metric $F$ and are given by,
 	\begin{equation}\label{eqn2.2.10}
 			{R}^{i}_{k}=2\frac{\partial{G}^i}{\partial	x^k}-\xi^j\frac{\partial^2{G}^i}{\partial x^j\partial	\xi^k}+2{G}^j\frac{\partial^2 {G}^i}{\partial \xi^j\partial	\xi^k}-\frac{\partial{G}^i}{\partial \xi^j}\frac{\partial {G}^j}{\partial \xi^k}.
 	\end{equation}
 	\end{defn}
 \noindent	The flag curvature $K=K(x,\xi,P)$, generalizes the sectional curvature in Riemannian geometry to the Finsler geometry and does not depend on whether one is using the Berwald, the Chern or the Cartan connection.

 	\begin{defn}[\textbf{Flag curvature} \cite{CXSZ}, $\S 4.1$] 
  \textnormal{	For a  plane $P\subset T_xM$ containing a non-zero vector $\xi$ called \textit{pole}, the \textit{flag curvature} $\textbf{K}(x,\xi,P)$ is defined by
 	\begin{equation}\label{eqn2.1.11}
 	\textbf{K}(x,\xi,P) :=\frac{g_\xi(R_\xi(u), u)}{g_\xi(\xi, \xi)g_\xi(u, u)- g_\xi(\xi, u)g_\xi(\xi, u)},
 	\end{equation}
 	where $u\in P$ is such that $P=\text{span} \left\{ \xi,u\right\} $.\\
 	 If $\textbf{K}(x,\xi,P)=\textbf{K}(x,\xi)$, then the	Finsler metric $F$ is said to be of scalar flag curvature and if $\textbf{K}(x,\xi,P)=constant$, then the Finsler metric $F$ is said to be of constant flag curvature .}
 		\end{defn}
The relation between the Riemann curvature  $R^i_j$ and the flag curvature $\textbf{K}(x,\xi)$ of a Finsler metric $F$ of scalar flag curvature is given by (see for more detail \cite{CXSZ}, $\S 4.1$)
 		 \begin{equation}\label{eqn2.1.12}
 		 R^i_j= \textbf{K}(x,\xi)\left\lbrace F^2 \delta^i_j-FF_{\xi^j}\xi^i\right\rbrace.
 		 \end{equation}
     
 \noindent It is well known that there is no canonical volume form on a Finsler manifold, like in the Riemannian case. Indeed, there are several well-known volume forms on the Finsler manifold, for instance,  the {\it Busemann-Hausdorff} volume form, the {\it Holmes-Thompson} volume form,  the {\it maximum } volume form,  the {\it minimum } volume form, etc. Here we discuss only the Busemann-Hausdorff volume form.
 \begin{defn}[Busemann-Hausdorff Volume form in Finsler manifolds \cite{SZ}, \boldmath$\S 2.2$]
 The \textit{Busemann-Hausdorff} volume form on a Finsler manifold $(M,F)$ is defined as:
  $dV_{BH}=\sigma_{BH}(x)dx$, where
 \begin{equation}
 \sigma_{BH}(x)=\frac{\emph{Vol} (B^n(1))}{\textnormal{Vol} \left\lbrace (\xi^i)\in T_xM: F(x,\xi)<1\right\rbrace}.
 \end{equation}
 \end{defn}

\vspace{0.1in}
\noindent
 The Busemann-Hausdorff volume form of the Randers metric can be explicitly given as follows:
 \begin{lem}[\cite{W}, \boldmath$\S 3$ ]\label{lem 3.A1}
 The Busemann-Hausdorff volume form of the Randers metric $F =\alpha + \beta$ is given by,
 \begin{equation}\label{eqn 3.A38}
 dV_{BH} =\sigma_{BH}(x)dV_\alpha= \left( 1-||\beta||^2_\alpha\right)^{\frac{n+1}{2}} dV_\alpha,
 \end{equation}
 where $dV_\alpha=\sqrt{\det (a_{ij})} \ dx$.
 \end{lem}

 \noindent For the \textit{Busemann-Hausdorff} volume form 
  $dV_{BH}=\sigma_{BH}(x)dx$ on Finsler manifold $(M,F)$, the \textit{distortion} $\tau$ is defined by (see, \cite{SSZ}, $\S 5.1$)
\begin{equation*}
\tau (x,\xi) :=\ln \frac{\sqrt{\det(g_{ij}(x,\xi))}}{\sigma _{BH}(x)}.
\end{equation*}
Now we define  $S$-curvature of the Finsler manifold $(M, F)$ with respect to the volume form $dV_{BH}$.
\begin{defn} [\textbf{S-curvature}, \cite{SSZ}, \boldmath$\S 5.1$]
 \textnormal{	For a vector $\xi\in T_xM\backslash \left\lbrace 0\right\rbrace$, let $\gamma=\gamma(t)$ be the geodesic with $\gamma(0)=x$ and $\dot{\gamma}(0)=\xi$.
 Then the $S$-curvature of the Finsler metric $F$ is defined by 	
\begin{equation*}
 S(x, \xi)=\frac{d}{dt}\left[ \tau \left(\gamma(t), \dot{\gamma}(t)\right) \right]|_{t=0}.
 \end{equation*}}
 \end{defn}
 The $S$-curvature of $F$ in  terms of spray coefficients is given by
 \begin{equation}\label{eqn2.1.15}
 S(x,\xi)=\frac{\partial G^m}{\partial \xi^m}-\xi^m\frac{\partial\left( \ln\sigma_{BH}\right) }{\partial x^m},
 \end{equation}
 where $G^m$ are given by  \eqref{eqn2.1.14}.

\begin{thm}\label{lem 3.1}
\textnormal{(\cite{DSSZ}, $\S 11.3$; \cite{SSZ}, $\S 3.4.8$)}
  If  $F=\alpha+\beta$  is a Randers metric on the manifold $M$ with $\beta$ a closed $1$-form, then the Finslerian geodesics have the same trajectories as the geodesics of the underlying Riemannian metric $\alpha$. Moreover, if  $(M, \alpha)$ has constant curvature, then $(M, F)$ is locally projectively flat and consequently, in this case 
 $(M, F)$  is  projectively equivalent to  $(M, \alpha)$. 
 \end{thm}
\begin{defn} [\text{Finsler structure for Funk metric in a bounded convex set} $\Omega \subset \mathbb{R}^n$  \cite{HHG}, \text{Chapter 3}, $\boldmath \S 3$]
    Let $\Omega \subset \mathbb{R}^n$ be a bounded convex set. The Finsler structure $F_{\Omega}$ for the Funk metric $d_{F,\Omega}$ on $\Omega$ is defined such that the unit ball in the tangent space $T_x\Omega$ at a point $x\in \Omega$ is the domain $\Omega$ itself, but with the point $x$ as the center. Thus, the  Finsler structure $F_{\Omega}$ is  defined as
    \begin{equation}\label{eqnn 4.2001}
     F_{\Omega}(x, \xi):= \inf  \left\lbrace t > 0 ~\vert  \left( x+\frac{\xi}{t} \right) \in   \Omega \right\rbrace,
    \end{equation}
    where $\xi \in T_x\Omega$.
       \end{defn}
    In the sequel, we restrict ourself to the discussion on dimension $2$ only.

\begin{proposition}\label{ppn4.2.1}
Let  $\mathbb{D}_E(r)=\left\lbrace x\in \mathbb{R}^2 : |x| < r   \right\rbrace \subset \mathbb{R}^2$ be the disc centred at origin and radius $r$ . Then the  Finsler structure for the Funk metric defined by \eqref{eqnn 4.2001} is given by,
 \begin{equation}
     F_{\mathbb{D}_E(r)}(x, \xi) =\frac{\sqrt{\left(r^2-|x|^2 \right) |\xi|^2+ \langle x , \xi \rangle ^2 }+ \langle x , \xi \rangle}{r^2-|x|^2}.
 \end{equation}
\end{proposition}
\begin{proof} From equation \eqref{eqnn 4.2001}, the Finsler metric on  $\mathbb{D}_E(r)$ is given by
    \begin{equation}\label{eqnn8}
     F_{{\mathbb{D}_E(r)}}(x, \xi)= \inf  \left\lbrace t > 0 ~\vert  \left( x+\frac{\xi}{t} \right) \in   \mathbb{D}_E(r) \right\rbrace, ~\forall~ \xi \in T_x\mathbb{D}_E(r).
    \end{equation}
    And therefore,
 \begin{eqnarray*}\label{eqnn9}
  \left( x+\frac{\xi}{F_{\mathbb{D}_E(r)}(x, \xi)} \right) \in  \partial \mathbb{D}_E(r), 
  \end{eqnarray*}
  $$\mbox{i.e.,}~~~~~~~~~~~~~ |x+\frac{\xi}{F_{\mathbb{D}_E(r)}(x, \xi)}|^2=r^2 ,$$
  equivalently~
  $$(r^2-|x|^2)F^2_{\mathbb{D}_E(r)}-2\langle x , \xi \rangle F_{\mathbb{D}_E(r)}-|\xi|^2 =0,$$ 
  After solving the above quadratic equation  for $F$ and keeping in mind that $F$ is non negative, we have
  \begin{equation}\label{eqnn10}
     F_{\mathbb{D}_E(r)}=\frac{\sqrt{\left(r^2-|x|^2 \right) |\xi|^2+ \langle x , \xi \rangle ^2 }+ \langle x , \xi \rangle}{r^2-|x|^2}. 
  \end{equation}
 See for more detail \cite{HHG}, Chapter 3, $\S 3$, Example $3.3$.
 \end{proof}

\vspace{.3cm}
\begin{rem}
 Clearly if $r=1$,  then the Finsler metric on  $\mathbb{D}_E(1)$ is given by
\begin{equation}\label{eqnn101}
     F_{\mathbb{D}_E(1)}(x, \xi) =\frac{\sqrt{\left(1-|x|^2 \right) |\xi|^2+ \langle x , \xi \rangle ^2 }+ \langle x , \xi \rangle}{1-|x|^2}.
 \end{equation}   
\end{rem}

 \section{The Funk metric on the Klein unit disc $\mathbb{D}_K(1)$}
The Hilbert geometry is a natural geometry defined in an arbitrary convex subset of real affine space. In 1894, Hilbert discovered how to associate a length to each segment in a convex set by way of an elementary geometric construction and using the cross-ratio. The special case where the convex set is a ball, or more generally an ellipsoid, gives the Beltrami-Klein model of hyperbolic geometry. In this sense, Hilbert's geometry is a generalization of hyperbolic geometry.\\
 In what follows, we refer \cite{HHG}, Chapter 3, $\S 4$ for more details on the Hilbert geometry. \\

\noindent
Let $\mathbb{D}_E(1)=\left\lbrace x\in \mathbb{R}^2 : |x| < 1 \right\rbrace$ be the unit Euclidean disc. 
Then the Klein/Hilbert distance on  $\mathbb{D}_E(1)$ is given by
\begin{equation}\label{eqn A}
d_K(x,y):=\frac{1}{2}\log \left(\frac{|x-a||y-b|}{|y-a||x-b|} \right), 
\end{equation}
where $x,y \in \mathbb{D}_E(1)$ and $a=\overrightarrow{xy} \cap \partial \mathbb{D}_E(1), b=\overleftarrow{xy} \cap \partial \mathbb{D}_E(1)$, \\
%see (\cite{HHG}, Chapter 3, $\S 4$).\\
The corresponding  Klein metric on  $\mathbb{D}_E(1)$ is the well-known Riemannian Klein  metric  and is given by 
 \begin{equation}\label{eqnn A}
  \alpha_K(x,\xi)=\frac{1}{2} \left( F_{\mathbb{D}_E(1)}(x, \xi)+ F_{\mathbb{D}_E(1)}(x, -\xi) \right)=\frac{\sqrt{\left(1-|x|^2 \right) |\xi|^2+ \langle x , \xi \rangle ^2 } }{1-|x|^2},
  \end{equation}
  where  $x=(x^1, x^2) \in \mathbb{D}_E(1), \xi \in T_x\mathbb{D}_E(1)$. %see (\cite{HHG}, Chapter 3, $\S 4$, Ex $4.2$). \\

\vspace{.3cm}
  \noindent
  In the sequel, we also denote $\alpha_K(x,\xi)$ by $||\xi||_K$.\\

  \noindent
It is well known that the Poincar\'e unit disc and Poincar\'e upper half-plane are two isometric models of hyperbolic plane geometry. The Poincar\'e metric $\alpha_U$ on the upper half plane $\mathbb{U}=\left\lbrace(x^1, x^2)\in \mathbb{R}^2 : x^2 > 0\right\rbrace,$   is  given by 
\begin{equation}\label{eqn 3.3}
\alpha_U(x,\xi)=\frac{|\xi|}{x^2},~\forall x=(x^1, x^2) \in \mathbb{U}, \xi=(\xi^1, \xi^2) \in T_x\mathbb{U}.
\end{equation}
\begin{proposition}\label{ppn 4.31}
    The Poincar\'e metric on the upper half plane $(\mathbb{U},\alpha_U)$ and the  Klein metric on the disc $(\mathbb{D}_E(1), \alpha_K)$ are Finsler isometric.
\end{proposition}
\begin{proof}
  Let $g:(\mathbb{D}_E(1), \alpha_K) \rightarrow (\mathbb{U}, \alpha_U)$ be the isometry given by (cf. \cite{AMAK}, Theorem 2),
\begin{equation*}\label{eqn4.A23}
g(x)=\left( \frac{2x^2}{1+x^1}, \frac{2\sqrt{1-|x|^2}}{1+x^1}\right).
\end{equation*}
Then
  $g^{-1}:\mathbb{U} \rightarrow \mathbb{D}_E(1) $ is,
\begin{equation}\label{eqn4.A24}
g^{-1}(x)=\left( \frac{4-|x|^2}{4+|x|^2}, \frac{4x^1}{4+|x|^2}\right).
\end{equation}
Clearly,
\begin{equation*}
g^*\alpha_U(x,\xi)=\alpha_U(g(x),dg_x(\xi))=\alpha_K(x,\xi), \end{equation*} 
that is, \begin{equation*}
g^*\alpha_U=\alpha_K. \end{equation*} \end{proof}
\begin{rem} It is interesting to note that, although the Klein disc $\mathbb{D}_E(1)$ and the Poincar\'e  upper half plane $\mathbb{U}$ are Finsler isometric to each other via the map $g$, i.e., the derivative map $dg$ between the corresponding tangent spaces is norm preserving. However, these two spaces are {\it not conformal} in the Riemannian sense, i.e., the angle between two vectors is not preserved, and therefore, they are not isometric in the Riemannian sense. 
\end{rem}
\vspace{0.1cm}

\begin{proposition}\label{ppn 4.32}
     The unit disc $\mathbb{D}_K(1)$ centered at origin with respect to  Klein distance $d_K$ is identical to the  Euclidean disc $\mathbb{D}_E(r)$ centered at origin and radius $r=\left( \frac{e^2-1}{e^2+1}\right)$, as a set.
\end{proposition}

\begin{proof}
 By definition, the Klein unit disc centered at origin is given by,
\begin{eqnarray}
 \nonumber \mathbb{D}_K(1)  &=&\left\lbrace x\in \mathbb{R}^2 : d_k(0,x) < 1  \right\rbrace\\
 \nonumber   &=& \left\lbrace x\in \mathbb{R}^2 : d_K(0,x)=\frac{1}{2}\log \left(\frac{|0-m||x-\Bar{m}|}{|x-m||0-\Bar{m}|} \right) < 1  \right\rbrace\\
  \nonumber   &=& \left\lbrace x\in \mathbb{R}^2 : \frac{|m||x-\Bar{m}|}{|x-m||\Bar{m}|}  < e^2  \right\rbrace, ~~ \mbox{as}~ |m|=\Bar{m}=1 \\
 \nonumber    &=& \left\lbrace x\in \mathbb{R}^2 :  \frac{|x-\Bar{m}|}{|x-m|}  < e^2  \right\rbrace.
    \end{eqnarray}
    Since $|x-m|=1-|x|$, $|x-\Bar{m}|=1+|x|$,
 \begin{equation}
  \nonumber    \mathbb{D}_K(1)=\left\lbrace x\in \mathbb{R}^2 : |x| < r   \right\rbrace = \mathbb{D}_E(r),~~ \mbox{where} ~r =\left( \frac{e^2-1}{e^2+1}\right),
 \end{equation}
where $\mathbb{D}_E(r)$ represents the Euclidean disc centered at origin and radius $r$.
\end{proof}
\vspace{.5 cm}
\noindent

\noindent The Finsler structure for the Funk metric in a convex domain $\Omega$ of the upper half plane $\mathbb{U}$ is given by Papadopoulos and Yamada in \cite{PAYS}. In view of Proposition \ref{ppn 4.31}, $(\mathbb{U},\alpha_U)$  and $(\mathbb{D}_E(1),\alpha_K)$  are isometric, i.e., norm preserving, therefore,  the Finsler structure  $\mathcal{F}(x, \xi)$ on the Klein unit disc $\mathbb{D}_K (1)$ is given by: 

\begin{equation}\label{eqnn 4.001}
\mathcal{F}(x, \xi)=\sup \limits_{\pi \in \mathscr{P} } \frac{\cosh d_K(x, T(x,\xi, \pi))}{\sinh d_K(x, T(x,\xi,\pi))} ||\xi||_K.
\end{equation}

\noindent Here $x \in  \mathbb{D}_k(1)$,  $\xi \in  T_{\mathbb{D}_k(1)}$, $\pi$ is a supporting hyperplane on  $\partial \mathbb{D}_k(1)$, $\mathscr{P}$ is the collection of all supporting hyperplanes on  $\partial \mathbb{D}_k(1)$ and $T(x, \xi, \pi)=R_{x,\xi} \cap \pi$ where, $R_{x,\xi}$ denotes  the geodesic ray starting from $x$ with initial tangent vector $\xi$.\\
Since $\coth t$ is a decreasing function of $t$ on the interval $(0,\infty )$,  \eqref{eqnn 4.001} reduces to
\begin{equation}\label{eqnn 4.001A}
  \mathcal{F}(x, \xi)=  \frac{\cosh d_K(x,\mathfrak{a})}{\sinh d_K(x,\mathfrak{a})} ||\xi||_K,
\end{equation}
where $\mathfrak{a}=R_{x,\xi} \cap \partial \mathbb{D}_K(1)$. 
\begin{thm}\label{thm0.3.1}
    The Finsler structure of the Funk metric on the Klein unit disc $\mathbb{D}_K(1)$ is a Randers metric, given by  $\mathcal{F} =\alpha_F +\beta_F, $ where   $\displaystyle \alpha_F (x, \xi) =\frac{\sqrt{\left(r^2-|x|^2 \right) |\xi|^2+ \langle x, \xi \rangle ^2 }}{r^2-|x|^2}$ with $r =\left( \frac{e^2-1}{e^2+1}\right)$, is the Riemannian-Klein metric in the disc $\mathbb{D}_K(1)$ and  $\beta$ is a closed $1$-form given by
$\displaystyle \beta_F (x, \xi) =\frac{(1-r^2) \langle x ,\xi \rangle}{(r^2-|x|^2)(1-|x|^2)}$, \ for all $x \in \mathbb{D}_K(1),\ \xi \in T_x\mathbb{D}_K(1)$  .
\end{thm}

\begin{proof}
 In view of \eqref{eqnn 4.001A} and \eqref{eqn A}, the Finsler structure of the Funk metric on the Klein unit disc $\mathbb{D}_K(1)$ is given by 
 \begin{eqnarray}
\label{16}  \mathcal{F}(x, \xi)=\frac{|x-m|\cdot|\mathfrak{a}-\bar{m}|+|x-\bar{m}|\cdot|\mathfrak{a}-m|}{|x-m|\cdot|\mathfrak{a}-\bar{m}|-|x-\bar{m}|\cdot|\mathfrak{a}-m|}||\xi||_K,
 \end{eqnarray}
 where $m=R_{x,\xi} \cap \partial \mathbb{D}_E(1), and 
  ~\bar{m}= R_{x, -\xi}  \cap \partial \mathbb{D}_E(1)$. Here $R_{x, -\xi}$ denotes the geodesic ray starting from $x$ with initial tangent vector $- \xi$. \\
  In view of of \eqref{eqnn 4.2001} we obtain,
\begin{equation}\label{4.3.A}
 x+\frac{\xi}{F_{\mathbb{D}_E(1)}(x,\xi)}=m,~ x-\frac{\xi}{F_{\mathbb{D}_E(1)}(x,-\xi)}=\bar{m},~ x+\frac{\xi}{F_{\mathbb{D}_K(1)}(x,\xi)}=\mathfrak{a}.
\end{equation} 
 Clearly, $F_{\mathbb{D}_K(1)}(x,\xi) > F_{\mathbb{D}_E(1)}(x,\xi)$. This yields that
 \begin{equation}\label{4.3.AA}
|\mathfrak{a}-m|=\frac{\left( F_{\mathbb{D}_K(1)}(x,\xi)-F_{\mathbb{D}_E(1)}(x,\xi)\right) |\xi|}{F_{\mathbb{D}_K(1)}(x,\xi)\cdot F_{\mathbb{D}_E(1)}(x,\xi)}, 
 \end{equation} 
 and 
  \begin{equation}\label{4.3.AAA}
 |\mathfrak{a}-\bar{m}|=\frac{\left( F_{\mathbb{D}_K(1)}(x,\xi)+F_{\mathbb{D}_E(1)}(x,-\xi)\right) |\xi|}{F_{\mathbb{D}_k(1)}(x,\xi)\cdot F_{\mathbb{D}_E(1)} (x,-\xi)}.
 \end{equation} 
Substituting \eqref{4.3.A}, \eqref{4.3.AA}, \eqref{4.3.AAA} in \eqref{16}, we have
 \begin{equation*}
     \mathcal{F}(x, \xi)=\frac{2F_{\mathbb{D}_k(1)}(x,\xi)-[F_{\mathbb{D}_E(1)}(x,\xi)-F_{\mathbb{D}_E(1)}(x,-\xi)]}{F_{\mathbb{D}_E(1)}(x,\xi)+F_{\mathbb{D}_E(1)}(x,-\xi)}||\xi||_K. 
 \end{equation*}
Further, considering   \eqref{eqnn10}, \eqref{eqnn101}  and \eqref{eqnn A}, we obtain 
 \begin{eqnarray}
 \nonumber \mathcal{F}(x, \xi) &=&\frac{\sqrt{\left(r^2-|x|^2 \right) |\xi|^2+ \langle x , \xi \rangle ^2 }}{r^2-|x|^2}+\frac{(1-r^2) \langle x ,\xi \rangle}{(r^2-|x|^2)(1-|x|^2)}\\
\label{eqn 4.31} &=& \alpha_F (x, \xi)+\beta_F(x, \xi).
 \end{eqnarray}
  If we write $\alpha_F (x, \xi)=\sqrt{a_{ij}\xi^i\xi^j}$ and $\beta_F(x, \xi)  =b_i(x)\xi^i$, then 
  
  \begin{equation}\label{eqn2.5.118}
   (a_{ij})=\frac{1}{(r^2-|x|^2)^2}
  \begin{pmatrix}
   r^2-|x|^2+(x^1)^2& x^1x^2 \\
  x^1x^2 & r^2-|x|^2+(x^2)^2
  \end{pmatrix},
  \end{equation}
   $\det(a_{ij})=\frac{r^2}{(r^2-|x|^2)^3}$ and its inverse matrix is given by 
  \begin{equation}\label{eqn2.5.119}
(a^{ij})=(a_{ij})^{-1}= \frac{(r^2-|x|^2)}{r^2 } 
  \begin{pmatrix}
     r^2-|x|^2+(x^2)^2& -x^1x^2 \\
    -x^1x^2 & r^2-|x|^2+(x^1)^2
    \end{pmatrix},
  \end{equation}
 and the coefficients $b_i(x)$ of the $1$-form $\beta_F$  is given by
 \begin{equation}\label{eqn2.5.120}
b_i(x)=\frac{(1-r^2) x^i }{(r^2-|x|^2)(1-|x|^2)}, 
 \end{equation}
 and hence \begin{equation}\label{eqn2.5.121}
||\beta_F||^2_{\alpha_F}=a^{ij}b_ib_j=\frac{|x|^2(1-r^2)^2}{r^2(1-|x|^2)^2}  < 1.
\end{equation}
\noindent It is easy to observe that   $\beta_F=df(x)$, where $f(x)=\frac{1}{2}\log\left(\frac{1-|x|^2}{r^2-|x|^2} \right).$
Thus $\mathcal{F}(x, \xi)=\alpha_F (x, \xi)+\beta_F(x, \xi)$  is a Randers metric with a closed $1$-form $\beta$.
\end{proof}

\begin{rem} The Finsler structure defined in 
     Theorem \ref{thm0.3.1} has same geodesics as geodesics of Riemannian Klein unit disc as a point set, by  Theorem \ref{lem 3.1}.  In particular, the geodesics as a point set are lines in the Finsler Klein disc described in Theorem \ref{thm0.3.1}. 
\end{rem}

\vspace{.5cm}
\noindent Now the distance induced by the Funk-Finsler structure $\mathcal{F}$ is given by
\begin{equation}\label{AS}
d(x, y)=\inf \limits_{\sigma} \int_{a}^{b} \mathcal{F}(\sigma(t), \dot{\sigma}(t)) \,dt ,
\end{equation}
 where the infimum is being taken over all  piecewise $C^1$-curves $\sigma$ in $\mathbb{D}_K(1)$ with $\sigma(a)=x$ and $\sigma(b)=y$.
 \begin{proposition}
 The Funk distance  $d_1(x, y)$, for the Klein unit disc $\mathbb{D}_K(1)$, can be recovered from the distance formula \eqref{AS} of its Finsler structure, where
 \begin{equation*}
d_1(x, y):=\log\frac{\sinh d_K(x,\mathfrak{a})}{\sinh d_K(y,\mathfrak{a})},
\end{equation*}
where $x,y \in \mathbb{D}_K(1)$ and  $\mathfrak{a}=\overrightarrow{xy} \cap \partial \mathbb{D}_K(1)$. 
 \end{proposition}
 \begin{proof}
 For a given pair of points $x$ and $y$ in $\mathbb{D}_K(1)$, let $\textnormal{exp}{_x} t\xi_{xy}$ be the Klein geodesic ray from $x$ through $y$. Let $\mathfrak{a}(x,y)=T(x, \xi_{xy}, \pi_{\mathfrak{a}(x,y)})$ be the point where this geodesic ray hits the boundary $\partial \mathbb{D}_K(1)$ (hitting in the sense of  Euclidean case). Then the   $F$-length of the curve of $\sigma$ from $x$ to $y$ is
 \begin{equation*}
 L(\sigma)= \int_{a}^{b} \mathcal{F}(\sigma(t), \dot{\sigma}(t)) \,dt =\log\frac{\sinh d_K(x,\mathfrak{a})}{\sinh d_K(y,\mathfrak{a})}.
 \end{equation*}
 Thus, $d(x, y)=d_1(x, y)$.
 \end{proof}

\subsection{The Realisation of the Funk metric on the Klein unit disc $\mathbb{D}_K(1)$ via the upper sheet of the hyperboloid of two sheets }

In this subsection, we  {\it construct} a non-positive definite Randers metric on the upper half space and show that its pullback on the upper sheet of the hyperboloid of two sheets is the realization of the Funk metric on the Klein unit disc.\\
  
 \noindent 
 Let $\mathbb{R}_+^3 = \left\lbrace(\tilde{x}^1,\tilde{x}^2, \tilde{x}^3)\in \mathbb{R}^3 : \tilde{x}^3 > 0\right\rbrace$, be the 
 the upper half space equipped with the Lorentzian  metric $\alpha_L$ defined by $\alpha_L(\tilde{x},\tilde{\xi})=  \sqrt{(\tilde{\xi}^1)^2+(\tilde{\xi}^2)^2-(\tilde{\xi}^3)^2}$ with $\tilde{x}\in \mathbb{R}_+^3 $ and $\tilde{\xi} \in T_{\tilde{x}}\mathbb{R}_+^3\cong \mathbb{R}^3$. 
 Now consider the deformation  of $\alpha_L$ by  the $1$-form
 $$\displaystyle \beta_L=\left(1-\frac{r^2}{(\tilde{x}^3)^2-r^2\left((\tilde{x}^1)^2+(\tilde{x}^2)^2\right)} \right) \frac{1}{\tilde{x}^3}d\tilde{x}^3$$ in $\mathbb{R}_+^3$ as follows: 
  \begin{equation}\label{eqn3.1222}
  F_L(\tilde{x},\tilde{v}) =\alpha_L(\tilde{x},\tilde{\xi})+\beta_L(\tilde{x},\tilde{\xi}).\end{equation}
  We call $F_L$ as  {\it Lorentz-Randers} metric as it is not {\it positive definite} metric on its domain.\\
\noindent  Now  we parametrize the upper half portion of the hyperboloid of two sheets $$\mathbb{H_+}=\left\lbrace(\tilde{x}^1,\tilde{x}^2, \tilde{x}^3)\in \mathbb{R}_+^3 : (\tilde{x}^3)=\sqrt{1+(\tilde{x}^1)^2+(\tilde{x}^2)^2} \right\rbrace$$
   in $\mathbb{R}^3$ as:
\begin{equation}\label{eqn 3.A2}
\eta : \mathbb{D}_E(r) \subset \mathbb{R}^2 \rightarrow \mathbb{H}_+\subset \mathbb{R}_+^3,~~~ \eta(x)=\left( \frac{x}{\sqrt{r^2-|x|^2}},\frac{r}{\sqrt{r^2-|x|^2}}\right).
\end{equation}
where $x=(x^1,x^2) \in \mathbb{D}_E(r)$. Note that $\eta$ is a smooth diffeomorphism between $\mathbb{D}_E(r)$ and $\mathbb{H_+}$. 
\begin{proposition}
The pullback of the Lorentz-Randers metric $F_L$ defined as above, on the upper sheet of the hyperboloid of two sheets by the map $\eta$ is the realization of the Finsler-Funk metric on the Klein unit disc in Klein disc model of hyperbolic geometry, that is,  $\eta^*F_L(x,\xi)=  \mathcal{F}(x,\xi)$ for all $(x,\xi) \in T \mathbb{D}_K(1)$. 
\end{proposition}
\begin{proof}
First we find $\eta^*F_L(x,\xi)$, for $x\in \mathbb{D}_K(1)$ and $\xi \in T_x \mathbb{D}_K(1) \cong \mathbb{R}^2$.
 We have by \eqref{eqn 3.A2},\\
  $$\eta^1(x) = \frac{x^1}{\sqrt{r^2-|x|^2}} , ~ \eta^2(x) = \frac{x^2}{\sqrt{r^2-|x|^2}}  \; \mbox{and} \; \eta^3(x) = \frac{r}{\sqrt{r^2-|x|^2}}.$$
 Therefore,
\begin{equation*}
d\eta^1_x=\frac{1}{(r^2-|x|^2)^\frac{3}{2}}\left[ \left\lbrace r^2-|x|^2+(x^1)^2\right\rbrace dx^1 + x^1x^2 dx^2\right],
\end{equation*}
\begin{equation*}
d\eta^2_x=\frac{1}{(r^2-|x|^2)^\frac{3}{2}}\left[  x^1x^2 dx^1 + \left\lbrace r^2 -|x|^2+ (x^2)^2\right\rbrace dx^2\right], 
\end{equation*}
\begin{equation*}
d\eta^3_x =\frac{r}{(r^2-|x|^2)^\frac{3}{2}}\left[  x^1 dx^1 +x^2 dx^2\right].
\end{equation*}
Consequently,
\begin{equation*}
\begin{split}
\eta^*\alpha_L(x,\xi)=\Big[ \sqrt{(d\eta^1_x)^2+(d\eta^2_x)^2-(d\eta^3_x)^2}\Big](\xi)\\
=\frac{\sqrt{\left(r^2-|x|^2 \right)|\xi|^2+\langle x , \xi \rangle ^2 }}{r^2-|x|^2}=\alpha_F(x,\xi),
\end{split}
\end{equation*}
and
\begin{eqnarray}
\nonumber    \eta^*\beta_L(x,\xi)&=&\Bigg[\left(1-\frac{r^2}{(\eta^3(x))^2-r^2\left((\eta^1(x))^2+(\eta^2(x))^2\right)} \right)\frac{1}{\eta^3(x)}d\eta^3_x\Bigg](\xi)\\
\nonumber &=&\frac{(1-r^2) \langle x ,\xi \rangle}{(r^2-|x|^2)(1-|x|^2)}=\beta_F(x,\xi).
\end{eqnarray}

Hence,
\begin{equation}\label{eqn4.301}
\begin{split}
\eta^*F_L(x,\xi)=\eta^*\alpha_L(x,\xi)+\eta^*\beta_L(x,\xi) =\mathcal{F}(x,\xi).
\end{split}
\end{equation}
\end{proof}

\begin{rem} 
It is interesting to note that the pullback of the Lorentz-Randers metric as obtained in  \eqref{eqn4.301} is indeed a \textit{legitimate positive definite} Randers metric. Although  Lorentz-Randers metric itself on $\mathbb{R}_+^3$  is not a positive metric.
\end{rem}

\subsection{The Realisation of the Funk metric on the Klein unit disc $\mathbb{D}_K(1)$ via the upper hemisphere with radius $r$ }\label{Sec 3.6}
Let $C= \left\lbrace(\tilde{x}^1,\tilde{x}^2, \tilde{x}^3)\in \mathbb{R}_+^3 : (\tilde{x}^1)^2+(\tilde{x}^2)^2 < 1\right\rbrace  $ be an open subset in the upper half space with the hyperbolic metric $\alpha_+$, defined by $\displaystyle \alpha_+(x)=\frac{\sqrt{(d\tilde{x}^1)^2+(d\tilde{x}^2)^2+(d\tilde{x}^3)^2}}{\tilde{x}^3}$ with $\tilde{x}=(\tilde{x}^1,\tilde{x}^2, \tilde{x}^3)\in C $ and $\tilde{\xi}=(\tilde {\xi}^1,\tilde{\xi}^2, \tilde{\xi}^3) \in T_{\tilde{x}}C\cong \mathbb{R}^3$. Now let us consider the deformation of the metric $\alpha_+$ by the one form $\displaystyle \beta_+=b(x)d\tilde{x}^3$, where $b(x)=\left(\frac{ |\tilde{x}|^2-1}{\tilde{x}^3\left( 1-(\tilde{x}^1)^2-(\tilde{x}^2)^2\right)}\right)$ as follows:
\begin{equation}\label{eqn 3.5}
F_+(\tilde{x},\tilde{\xi})=\alpha_+ (\tilde{x},\tilde{\xi})+ \beta_+(\tilde{x},\tilde{\xi}) = \frac{|\tilde{\xi}|}{\tilde{x}^3}+ b(x)\tilde{\xi},
\end{equation}
where $\tilde{x}=(\tilde{x}^1,\tilde{x}^2, \tilde{x}^3)\in \mathbb{R}_+^3 $ and $\tilde{\xi}=(\tilde {\xi}^1,\tilde{\xi}^2, \tilde{\xi}^3) \in T_{\tilde{x}}\mathbb{R}_+^3\cong \mathbb{R}^3$. It is easy to show that $F_+$ is a Randers metric on the set $C$. \\

\noindent
Let us consider the map 
\begin{equation*}
\Psi:\mathbb{D}_K(1) \subset \mathbb{R}^2  \rightarrow \mathbb{S}_+^2=\left\lbrace(\tilde{x}^1,\tilde{x}^2, \tilde{x}^3)\in \mathbb{R}_+^3 : (\tilde{x}^1)^2+(\tilde{x}^2)^2 +(\tilde{x}^3)^2 = r^2\right\rbrace    
\end{equation*}
 given by 
 \begin{equation*}\label{eqn 3.A4}
\Psi(x^1, x^2)=\left(x^1, x^2, \sqrt{r^2-|x|^2} \right),
\end{equation*}
 which represents the immersion of the upper hemisphere in $C$. It is well known that the pullback of the hyperbolic metric on the upper sphere by the map $\Psi$ is the well known Klein metric on the Klein unit disc $\mathbb{D}_K(1)$. In fact, here we show that the pullback of the Randers metric \eqref{eqn 3.5} on the upper hemisphere is the realization of the Funk metric on the Klein unit disc $\mathbb{D}_K(1)$.

\vspace{.3cm}
\noindent
\begin{proposition}
 The pullback of the metric $F_+$ defined as above, on the upper hemisphere with radius $r$ by the map $\Psi$ is the realization of the Funk metric on the Klein unit disc $\mathbb{D}_K(1)$, that is,  $\Psi^*F_+(x,\xi)= \mathcal{F}(x,\xi) $ for all $(x, \xi) \in T\mathbb{D}_K(1)$.
\end{proposition}
\begin{proof} In view of  \eqref{eqn 3.5} and \eqref{eqn 3.A4}, we have for all $ (x, \xi) \in T\mathbb{D}_K(1)$,

\begin{eqnarray}
\nonumber    \Psi^*F_+(x,\xi)&=&F_+(\Psi(x), d\Psi_x(\xi))\\
\nonumber  &=&\frac{\sqrt{\left(r^2-|x|^2 \right) |\xi|^2+ \langle x , \xi \rangle ^2 }}{r^2-|x|^2}+\frac{(1-r^2) \langle x ,  \xi \rangle}{(r^2-|x|^2)(1-|x|^2)}\\
\nonumber  &=& \mathcal{F}(x,\xi).
\end{eqnarray}
\end{proof}
\subsection{The Funk-Finsler metric on the Poincar\'e unit disc in the Poincar\'e  disc $\mathcal{D}_{P}$}
In this section, 
with the help of the pullback of the Funk structure of the unit disc $\mathbb{D}_K(1)$ in the Klein
disc model $\mathcal{D}_{K}$, 
we define the  Funk structure on the unit disc 
$\mathbb{D}_P(1)$
in the Poincar\'e   disc model $\mathcal{D}_{P}$. 
It is well known that the Klein disc model $\mathcal{D}_{K}$ and  Poincar\'e disc model $\mathcal{D}_{P}$ of the hyperbolic geometry are Finslerian isometric, i.e., norm preserving via the diffeomorphism $f: \mathcal{D}_{P} \rightarrow \mathcal{D}_{K}$ given by  $f(x)=\frac{2x}{1+|x|^2}$, for all  $x\in \mathcal{D}_P$ (see \cite{AMAK}, $\S 4$, Theorem 1). The following proposition is in order.
\begin{proposition}\label{ppn 4.32A}
     The unit disc $\mathbb{D}_P(1)$ centered at origin with respect to   the Poincar\'e distance $d_P(p,q)=\cosh^{-1}\left(1+ \frac{2|p-q|^2}{(1-|p|^2)(1-|q|^2)} \right)$ for all $p,q \in \mathcal{D}_{P}$ is identical to the  Euclidean disc $\mathbb{D}_E(r')$ centered at origin and radius $r'=\left( \frac{e-1}{e+1}\right)$, as a set.
\end{proposition}
\begin{proof}
    The proof of $\mathbb{D}_P(1)=\mathbb{D}_E(r')$  is trivial.
\end{proof}
\begin{thm}
    The Funk structure on the Poincar\'e unit disc $\mathbb{D}_P(1)$ in the  Poincar\'e disc model $\mathcal{D}_P$ is given by 
    \begin{equation}\label{eqnn 3.2AS1}
    F_P(x,\xi):=  \alpha_{P}(x, \xi) + \beta_{P}(x, \xi)
\end{equation}
where 
\begin{equation*}
\alpha_{P} (x, \xi)=\frac{2 \sqrt{(1-|x|^2)^2|\xi|^2-(1-r^2)\Big[(1+|x|^2)^2|\xi|^2-4\langle  x , \xi \rangle^2\Big]}}{(1-|x|^2)^2-(1-r^2)(1+|x|^2)^2},
\end{equation*}

\begin{equation*}
\beta_{P} (x, \xi)=
\frac{4(1-r^2)(1+|x|^2)\langle  x , \xi \rangle}{(1-|x|^2)\left( (1-|x|^2)^2-(1-r^2)(1+|x|^2)^2\right)},
\end{equation*}
here $x\in \mathbb{D}_P(1), ~\xi \in T_x\mathbb{D}_P(1)$.
\end{thm}

\begin{proof}
Using the fact that $f:  \mathcal{D}_{P} \rightarrow    \mathcal{D}_{K}$, given by
 $f(x)=\frac{2x}{1+|x|^2}$, for all $x\in \mathcal{D}_P$ 
 is a diffeomorphism, it is easy to see that
$f(\mathbb{D}_P(1)) =\mathbb{D}_K(1)$.
The Funk structure $F_P$ on $\mathbb{D}_P(1)$ can be obtained  by pullback of the Funk structure $\mathcal{F}$ on $\mathbb{D}_K(1)$, given by  \eqref{eqn 4.31},  by the map $f$, i.e., 
    \begin{equation}\label{eqn AS43}
    F_P(x,\xi):= f^*\mathcal{F}(x,\xi)=\mathcal{F}(f(x), df_x(\xi))=\alpha_F(f(x), df_x(\xi))+\beta_F(f(x), df_x(\xi)), 
    \end{equation}
where $  (x,\xi) \in T\mathbb{D}_P(1)$.\\
Given a point $x=(x^1,x^2)\in \mathbb{D}_P(1)$,  $df_x$ is determined by the Jacobian 
    \begin{equation*}
        Jf(x)=\frac{2}{(1+|x|^2)^2}\begin{pmatrix}
 1+|x|^2-2(x^1)^2 & -2x^1x^2 \\
 -2x^1x^2 & 1+|x|^2-2(x^2)^2
 \end{pmatrix},
    \end{equation*}
then for every $\xi=(\xi^1, \xi^2) \in T_x\mathbb{D}_P(1) \cong \mathbb{R}^2$, we have 
 \begin{equation*}
        Jf(x)(\xi)=\frac{2}{(1+|x|^2)^2}\begin{pmatrix}
 \xi^1(1+|x|^2)-2x^1 \langle  x , \xi \rangle \\ \xi^2(1+|x|^2)-2x^2 \langle  x , \xi \rangle
 \end{pmatrix}.
      \end{equation*}
      Note that,
\begin{equation}\label{eqnn A51}
  1-|f(x)|^2=  \left( \frac{1-|x|^2}{1+|x|^2}\right)^2,
\end{equation}
\begin{equation}\label{eqnn A52}
    |df_x(\xi)|^2=\frac{4}{(1+|x|^2)^4}\left( (1+|x|^2)^2 |\xi|^2-4 \langle  x , \xi \rangle^2 \right).
\end{equation}
And
\begin{equation}\label{eqnn A53}
  \langle  f(x) , df_x(\xi) \rangle=  \frac{4}{(1+|x|^2)^3}\left( 1-|x|^2 \right)\langle  x , \xi \rangle.
\end{equation}
In view of \eqref{eqn 4.31}, \eqref{eqn AS43} and \eqref{eqnn A51}-\eqref{eqnn A53}, we have 

\begin{eqnarray}
\nonumber  F_P(x,\xi)&=&   \alpha_{P}(x, \xi) + \beta_{P}(x, \xi).
\end{eqnarray}
\end{proof}
\subsection{The Funk-Finsler metric on the unit disc in the Poincar\'e  upper half-plane}
In this subsection, we find the Funk structure on the disc of unit hyperbolic radius in the Poincar\'e upper half-plane. Also the Poincar\'e upper half-plane, Klein disc  and Poincar\'e  disc  are Finsler isometric to each other. The reasonable disc for the study will be the image of the  Poincar\'e unit disc or Klein unit disc via the respective isometries. It is easy to see that the unit hyperbolic disc centered at point $(0,2)$, (say $a$) in the Poincar\'e upper half-plane, which we denote by  $\mathbb{D}_U(a,1)$ is the image of the  Poincar\'e unit disc or Klein unit disc 
via the respective isometries.
The following proposition is in order:
\begin{proposition}\label{ppn 4.32AA}
     The unit hyperbolic disc $\mathbb{D}_U(a,1)$ centered at $a=\left( 0,2\right)$ is identical to the  Euclidean disc $\mathbb{D}_E\left(b, \Tilde{r}\right)$,  centered at $b=\left( 0,e+\frac{1}{e}\right)$  and radius $\Tilde{r}=e-\frac{1}{e}$, as a set.
\end{proposition}

\begin{proof} In view of the  fact that the hyperbolic distance in the upper half plane between two points $p=(p^1,p^2)$ and $ q=(q^1,q^2)\in U$ is given by
 $d_P(p,q)=\cosh^{-1}\left(1+ \frac{|p-q|^2}{2 p^2 q^2 } \right)$,  it is easy to show that as point set $\mathbb{D}_U(a,1)=\mathbb{D}_E\left(b, \Tilde{r}\right)$.
\end{proof}

\begin{thm}
    The Funk structure on the Poincar\'e unit disc $\mathbb{D}_U(a,1)$ in the  Poincar\'e upper half model $U$ is given by 
    \begin{equation}\label{eqnn 3.2AS112}
    F_U(x,\xi):=  \alpha_{U}(x, \xi) + \beta_{U}(x, \xi),
\end{equation}
where 

\begin{equation*}
\alpha_{U} (x, \xi)=
\frac{4\sqrt{16(x^2)^2|\xi|^2-(1-r^2)\Big[ 16\langle  x , \xi \rangle^2+\left\lbrace (4+|x|^2)\xi^1-2x^1\langle  x , \xi \rangle\right\rbrace^2 \Big]}}{16(x^2)^2-(1-r^2)(4+|x|^2)^2},
\end{equation*}

and 
\begin{equation*}
\beta_{U} (x, \xi)=\frac{(1-r^2)(4+|x|^2)\Big[x^1\xi^1(4+|x|^2)-(4-|x|^2+2(x^1)^2)\langle  x , \xi \rangle\Big]}{(x^2)^2\Big[16 (x^2)^2-(1-r^2)(4+|x|^2) \Big]},
\end{equation*}
here $x\in \mathbb{D}_U(a,1),~\xi\in T_x\mathbb{D}_U(a,1)$.

\end{thm}
\begin{proof}
Consider the diffeomorphism $g^{-1}: \mathbb{U} \rightarrow  \mathcal{D}_K$ given by \eqref{eqn4.A24} i.e, $g^{-1}(x)=\left( \frac{4-|x|^2}{4+|x|^2}, \frac{4x^1}{4+|x|^2}\right)$. And it is easy to see that
$g^{-1}(\mathbb{D}_U(a,1)) =\mathbb{D}_K(1)$.
Now from \eqref{eqn 4.31}, the Funk-Finsler metric on $\mathbb{D}_K(1)$ is  $\mathcal{F}(x, \xi)$.
We define, 
\begin{eqnarray}
   \nonumber  F_U(x,\xi)&:=& (g^{-1})^*\mathcal{F}(x,\xi)=\mathcal{F}(g^{-1}(x), dg^{-1}_x(\xi))\\
 \label{eqn 4. A57}    &=&\alpha_F(g^{-1}(x), dg^{-1}_x(\xi))+\beta_F(g^{-1}(x), dg^{-1}_x(\xi)),
\end{eqnarray}
where $  (x,\xi) \in T\mathbb{D}_U( a,1)$.\\
Given a point $x=(x^1,x^2)\in \mathbb{D}_U(a,1)$,  $dg^{-1}_x$ is given by
    \begin{equation*}
        dg^{-1}_x(\xi)=\frac{4}{(4+|x|^2)^2}\begin{pmatrix}
 -4 \langle  x , \xi \rangle \\  (4+|x|^2)\xi^1 -2x^1 \langle  x , \xi \rangle 
 \end{pmatrix},
      \end{equation*}
where $\xi=(\xi^1, \xi^2) \in T_x\mathbb{D}_U(a,1) \cong \mathbb{R}^2$. 
 
      Note that,
\begin{equation} \label{eqn 4. A59} 
  1-|g^{-1}(x)|^2=   \frac{16(x^2)^2}{(4+|x|^2)^2},
\end{equation}
\begin{equation} \label{eqn 4. A60} 
    |dg^{-1}_x(\xi)|^2=\frac{16\Big[16 \langle  x , \xi \rangle^2 +  \Bigl \{(4+|x|^2)\xi^1 -2x^1\langle  x , \xi \rangle \Bigl \}^2\Big]}{(4+|x|^2)^4},
\end{equation}
and
\begin{equation} \label{eqn 4. A61} 
  \langle  g^{-1}(x) , dg^{-1}_x(\xi) \rangle=  \frac{16\Big[x^1\xi^1\left( 4+|x|^2 \right)-\left( 4-|x|^2+2(x^1)^2 \right)\langle  x , \xi \rangle\Big]}{(1+|x|^2)^3}.
\end{equation}

In view of \eqref{eqn 4.31}, \eqref{eqn 4. A57} and \eqref{eqn 4. A59}-\eqref{eqn 4. A61},    we obtain
\begin{eqnarray}
\nonumber  F_U(x,\xi)= \alpha_{U}(x, \xi) + \beta_{U}(x, \xi).
\end{eqnarray}

\end{proof}

\section{Some curvatures of the Funk-Finsler structure on the Klein unit disc}
In this section, we explicitly obtain the expressions for the  $S$-curvature, Riemann curvature, flag curvature, and the Ricci curvature
of the Funk-Finsler structure $\mathcal{F}(x,\xi)$ on Klein unit disc  $\mathbb{D}_K(1)$. 

\subsection{The $\textbf{S}$-curvature of the Funk-Finsler metric on the Klein  unit disc}
In this subsection, we recall the formula for $S$-curvature of 
a general Randers metric $F=\alpha+\beta$, where $\alpha(x, \xi)=\sqrt{a_{ij}\xi^i\xi^j}$ and $\beta(x, \xi)=b_i(x)\xi^i$.\\
Let  $\bar{\Gamma}^k_{ij}(x)$  denote the  Christoffel symbols of Riemannian metric $\alpha$. Then we have,
 \begin{equation}\label{eqnn 4.3.48}
 b_{i|j}:=\frac{\partial b_i}{\partial x^j}-b_k \bar{\Gamma}^k_{ij}.
 \end{equation}
 We introduce the following notations,
 \begin{equation}\label{eqnn 4.3.50}
  r_{ij}:=\frac{1}{2}\left( b_{i|j}+b_{j|i}\right),~~ s_{ij}:=\frac{1}{2}\left( b_{i|j}-b_{j|i}\right).
  \end{equation}
 \begin{equation}\label{eqnn 4.3.51}
   s^i_j:=a^{ih}s_{hj},~~ s_j:=b_i s^i_j=b^j s_{ij}, ~~r_j:=b^i r_{ij}, ~~b^j=a^{ij} b_i,
   \end{equation}
\begin{equation}\label{eqnn 4.3.52}
e_{ij}:=r_{ij}+b_is_j+b_js_i,
\end{equation}
\begin{equation*}e_{00}:=e_{ij}\xi^i\xi^j,~~ s_0:=s_i\xi^i ~~ \text{and}~~ s^i_0:=s_j^i\xi^j.\end{equation*}\\
Now consider, 
 \begin{equation}\label{eqnn 4.3A}
      \rho:=\log \sqrt{\left( 1-||\beta||^2_\alpha\right)},~\mbox{and}~ \rho_0:=\rho_i\xi^i,\rho_i:=\rho_{x^i}(x).
  \end{equation}
It is well known that the $S$-curvature of the Randers metric $F=\alpha +\beta $ is given by,
\begin{equation}\label{eqnn 4.65}
 S=(n+1)\Big[\frac{e_{00}}{2F}-(s_0+\rho_0) \Big],
 \end{equation}
see \cite{CXSZ},  $\S 3.2$ for more details.

\vspace{.5cm}
\begin{thm}
    The $S$-curvature of the Funk-Finsler metric $\mathcal{F}(x,\xi)$, 
    given by \eqref{eqn 4.31}, on the Klein unit disc $\mathbb{D}_K(1)$  is given by:
    \begin{equation*}
     \textbf{S}(x,\xi)= \frac{3(1-r^2)\left[ (1-|x|^2)|\xi|^2+2\langle x , \xi \rangle^2\right] }{2\mathcal{F} (r^2-|x|^2)(1-|x|^2)^2}-\frac{3\langle x , \xi \rangle (1-r^2)^2(1+|x|^2)}{(r^2-|x|^2)(1-r^2|x|^2)(1-|x|^2)}  
    \end{equation*}
 where $(x,\xi) \in T\mathbb{D}_K(1)$.
\end{thm}
\begin{proof}
From equation \eqref{eqnn 4.65}, to calculate the $S$-curvature of  Randers metric given by
\eqref{eqn 4.31}, we proceed as follows.
The Christoffel symbols $\bar{\Gamma}^k_{ij}(x)$ of Riemannian Klein metric $\alpha_F$  are given by
 \begin{eqnarray}
\label{eqnn 4.3.49}  \bar{\Gamma}^k_{ij}(x):=\frac{x^i \delta_{kj}+x^j \delta_{ki}}{(r^2-|x|^2)},~~ \mbox{for all}~~ x\in \mathbb{D}_K(1).
  \end{eqnarray}
   Clearly,
  \begin{equation}\label{eqnn 4.3.490} 
 \bar{\Gamma}^k_{ij}(x)=\bar{\Gamma}^k_{ji}(x).
  \end{equation}
 Using expression for $b_i$ from \eqref{eqn2.5.120}, we get,
   \begin{equation}\label{eqnn 4.3.520}
    \frac{\partial b_i}{\partial x^j}=\frac{\partial b_j}{\partial x^i}=\frac{(1-r^2)}{(r^2-|x|^2)(1-|x|^2)}\Big[ \delta_{ij}+\frac{2x^ix^j(1+r^2-2|x|^2)}{(r^2-|x|^2)(1-|x|^2)} \Big].
\end{equation}
Substituting \eqref{eqnn 4.3.490} and \eqref{eqnn 4.3.520}
in \eqref{eqnn 4.3.48}, we obtain
\begin{equation}\label{eqnn 4.3.521}
   b_{i|j} =b_{j|i}=\frac{(1-r^2)}{(r^2-|x|^2)(1-|x|^2)}\Big[ \delta_{ij}+\frac{2x^ix^j}{(1-|x|^2)} \Big].
\end{equation}
Using \eqref{eqnn 4.3.521} in \eqref{eqnn 4.3.50}  yields,
\begin{equation}\label{eqnn 4.3.510}
  s_{ij}=0,~ r_{ij}=b_{i|j}=\frac{(1-r^2)}{(r^2-|x|^2)(1-|x|^2)}\Big[ \delta_{ij}+\frac{2x^ix^j}{(1-|x|^2)} \Big] ,~\forall i,j=1,2.
   \end{equation}
Applying  \eqref{eqnn 4.3.510} in \eqref{eqnn 4.3.51} and \eqref{eqnn 4.3.52}, we have
 \begin{equation}\label{eqnn 4.3.53}
  s^i_j:=0,~ s_j:=0, ~ \mbox{and}~~ e_{ij}=r_{ij}=\frac{(1-r^2)\left[ (1-|x|^2)\delta_{ij}+2x^ix^j\right] }{(r^2-|x|^2)(1-|x|^2)^2}.
 \end{equation}
Let $G^i=G^i(x,\xi)$ and $\bar{G}^i=\bar{G}^i(x,\xi)$ denote the spray coefficients of $\mathcal{F}$ and $\alpha$ respectively. Then  $G^i$ and $\bar{G}^i$ are related by
\begin{equation}\label{eqnn 4.54}
G^i=\bar{G}^i+P \xi^i+Q^i,
\end{equation}  
where
\begin{equation}\label{eqnn 4.55}
P:=\frac{e_{00}}{2\mathcal{F}}-s_0, ~~ ~~  Q^i:=\alpha s^i_0  ~~ \mbox{and}~~ \bar{G}^i=\frac{1}{2}\bar{\Gamma}^i_{jk}\xi^j\xi^k,
\end{equation}
and  where $e_{00}:=e_{ij}\xi^i\xi^j$, $s_0:=s_i\xi^i$ and $s^i_0:=s_j^i \xi^j$.\\
 Therefore from \eqref{eqnn 4.3.49} , \eqref{eqnn 4.3.53} and \eqref{eqnn 4.55} we find  that,
 \begin{equation}\label{eqnn 4.3.56}
 P=\frac{e_{00}}{2\mathcal{F}}=\frac{r_{ij}\xi^i\xi^j}{2\mathcal{F}}= \frac{(1-r^2)\left[ (1-|x|^2)|\xi|^2+2\langle x , \xi \rangle^2\right] }{2\mathcal{F} (r^2-|x|^2)(1-|x|^2)^2},
 \end{equation}
 \begin{equation}\label{eqnn 4.3.57}
 Q^i=0 ~ \mbox{and}~~ \bar{G}^i=\frac{\xi^i\langle x , \xi \rangle}{(r^2-|x|^2)},~  i=1,2.
 \end{equation}
 Employing \eqref{eqnn 4.3.56}, \eqref{eqnn 4.3.57} in  \eqref{eqnn 4.54}, we see that the spray coefficients are given by
 \begin{equation*}\label{eqnn 4.3.58}
 G^i=\frac{\xi^i}{(r^2-|x|^2)}\left(\langle x , \xi \rangle+\frac{(1-r^2)\left[ (1-|x|^2)|\xi|^2+2\langle x , \xi \rangle^2\right]}{2\mathcal{F}(1-|x|^2)^2} \right). 
 \end{equation*}
%For $n=2$,  using \eqref{eqn 3.A38}, %\begin{equation}\label{eqnn 4.61}
  %    \sigma_{BH}(x)=e^{3\rho(x)}\sigma_\alpha(x), 
%\end{equation}
%where $\rho$ is given by \eqref{eqnn 4.3A}.
%From \eqref{eqnn 4.54} and \eqref{eqnn 4.55}, we obtain
%  \begin{equation}\label{eqnn 4.62}
%\frac{\partial G^i}{\partial \xi^i}=\frac{\partial G^i_\alpha}{\partial \xi^i}+3P,
%\end{equation}
% where $P$ is given by \eqref{eqnn 4.3.56}.  Since $\alpha$ is Riemannian,  
 %by Example $3.1.1$ of \cite{CXSZ}, we have\\
 %\begin{equation}\label{eqnn 4.63}
%   \frac{\partial G^i_\alpha}{\partial \xi^i}-%\xi^i\frac{\partial\left( \ln\sigma_{\alpha}\right) }%{\partial x^i}=0.
 %\end{equation}
Availing \eqref{eqn2.5.121} , \eqref{eqnn 4.3A} and \eqref{eqnn 4.3.56} in \eqref{eqnn 4.65} for $n=2$, we obtain $S$-curvature as:
 \begin{equation*}
     \textbf{S}(x,\xi)= \frac{3(1-r^2)\left[ (1-|x|^2)|\xi|^2+2\langle x , \xi \rangle^2\right] }{2\mathcal{F} (r^2-|x|^2)(1-|x|^2)^2}-\frac{3\langle x , \xi \rangle (1-r^2)^2(1+|x|^2)}{(r^2-|x|^2)(1-r^2|x|^2)(1-|x|^2)}  
    \end{equation*}
\end{proof}

 %\begin{corollary}
% Let $x(t) = (x^{1}(t), x^{2}(t))$  be a geodesic in 
%$(\mathbb{D}_K(1), \mathcal{F})$. Then it satisfies the  %differential equation 
%\begin{eqnarray}\label{AA}
%\ddot{x}^i(t)&+&\frac{2(\dot{x}(t))^i}{(r^2-|x(t)|^2)}\Bigg[ %\langle x(t) , \dot{x}(t) \rangle+\\
%\nonumber &~&~~~~~~~~~~~~~~~~~~\frac{(1-r^2) \Big[ (1-%|x(t)|^2)|\dot{x}(t)|^2+2\langle x(t) , \dot{x}(t) \rangle^2 %\Big]}{2\mathcal{F}(x(t), \dot{x}(t))(1-|x(t)|^2)^2}\Bigg]=0.
% \end{eqnarray}
%The lines in $(\mathbb{D}_K(1), \mathcal{F})$ with %appropriate parameterization
%satisfy the above equation and hence are geodesics.
% \end{corollary}

%\begin{proof}
%In view of \eqref{eqn2.1.13} and \eqref{eqnn 4.3.58},  we %get the required result.
 %  Using \eqref{eqnn 4.3.58} in \eqref{eqn2.1.13},  the %geodesic equations 
 %  satisfied by $x(t)$ are given by \eqref{AA}.
 %  The second statement follows as $(\mathbb{D}_K(1), \mathcal{F})$
 %  is a projectively flat metric with Riemannian Klein %metric and the Riemannian geodesics of $(M, \alpha_F)$ are %lines (see Theorem \ref{lem 3.1}).
%\end{proof}

\vspace{.5 cm}
\begin{corollary}
The Funk-Finsler metric $\mathcal{F}$ on the Klein unit disc $\mathbb{D}_K(1)$  is a non-Berwaldian Douglas metric. 
 \end{corollary}
\begin{proof}
    It is known that a Randers metric $F=\alpha+\beta$ is Berwald if and only if $\beta$ is parallel with respect to $\alpha$, i.e., $b_{i|j}=0$ (see \cite{SSZ}, Lemma $3.1.2$) and $F$ is Douglas if and only if $b_{i|j}=b_{j|i}$ (or $s_{ij}=0$, see \cite{CXSZ}, Theorem $5.2.1$).
    In view of \eqref{eqnn 4.3.510} clearly  $\mathcal{F}$ is Douglas.\\
    If possible let $\mathcal{F}$ is Berwald,  i.e., $b_{i|j}=0$. Since $|x|^2 < r < 1$, we have
    \begin{equation*}
        \delta_{ij}+\frac{2 x^i x^j}{1-|x|^2}=0.
    \end{equation*}
    Put $i=j$ and summing over $i$, we get a contradiction, and hence $\mathcal{F}$ is a non-Berwaldian Douglas metric.
    \end{proof}

\begin{thm}
Let  $\mathcal{F}$ be the Funk metric on the Klein unit disc $\mathbb{D}_K(1)$  given by \eqref{eqn 4.31}. Then the Riemann curvature $R^i_k$, the Ricci curvature  $\textbf{Ric}$, and  the flag curvature $\textbf{K}$ of  $\mathcal{F}$ are given by
\begin{equation}
 R^i_k=- \left(  \delta^i_k \alpha_F^2-\alpha_F (\alpha_F)_k\xi^i \right)+\Bigg[ 3\left( \frac{\phi}{2\mathcal{F}}\right)^2  -\frac{\psi}{2\mathcal{F}}\Bigg]\left( \delta^i_k-\frac{\mathcal{F}_{\xi^k}}{\mathcal{F}}\xi^i\right) +\tau_k \xi^i,
 \end{equation}
and 
 \begin{equation}\label{eqn 4.4.2A}
\textbf{Ric} = \textbf{K} \mathcal{F}^2 =\Bigg[3\left( \frac{\phi}{2\mathcal{F}}\right)^2  -\frac{\psi}{2\mathcal{F}} \Bigg]-\alpha_F^2,
\end{equation}

where
$$ \psi=\frac{2(1-r^2)\langle x , \xi \rangle}{(1-|x|^2)^3(r^2-|x|^2)^2}\Bigg[ (1-|x|^2)|\xi|^2(3r^2-2|x|^2-1)-2 \langle x ,\xi \rangle^2(1+|x|^2-2r^2) \Bigg],$$
$$\phi=\frac{(1-r^2)\left[ (1-|x|^2)|\xi|^2+2\langle x , \xi \rangle^2\right] }{ (r^2-|x|^2)(1-|x|^2)^2}.$$
and 
\begin{equation*}
    \tau_k=\frac{(1-r^2)(x^k |\xi|^2-\xi^k \langle x , \xi \rangle)}{\mathcal{F} (r^2-|x|^2)^2(1-|x|^2)}.
\end{equation*}
 \end{thm}
 \begin{proof} The Riemannian curvature of the Randers metric $F=\alpha+\beta$ with a closed $1$-form on an $n$-dimensional manifold is a given by (see \cite{CXSZ}, $\S 5.2$, equation $(5.10)$ and equation $(5.12)$)
 \begin{equation}\label{eqnn 4.3.70}
 R^i_k=\overline{R^i_k}+\Bigg[ 3\left( \frac{\phi}{2F}\right)^2  -\frac{\psi}{2F}\Bigg]\left( \delta^i_k-\frac{F_{\xi^k}}{F}\xi^i\right) +\tau_k \xi^i,
 \end{equation}
where
\begin{equation}\label{eqnn 4.3.72}
\phi:=b_{i|j}\xi^i\xi^j,~~~\psi:=b_{i|j|k}\xi^i\xi^j\xi^k,~~~\tau_k:=\frac{1}{F}\left( b_{i|j|k}-b_{i|k|j}\right) \xi^i\xi^j,
\end{equation}
and 
\begin{equation}\label{eqnn 4.3.720}
 b_{i|j|k}=\frac{\partial b_{i|j}}{\partial x^k}-b_{i|m}\bar{\Gamma}^m_{jk}-b_{j|m}\bar{\Gamma}^m_{ik}.   
\end{equation}
Here $\overline{R^i_k}$   denotes the Riemann curvature of the Riemannian metric $\alpha$.\\
Further,  the Ricci curvature of the Randers metric is given by
\begin{equation}\label{eqnn 4.3.71}
\textbf{Ric} =\overline{\textbf{Ric}} +(n-1)\Bigg[3\left( \frac{\phi}{2F}\right)^2  -\frac{\psi}{2F} \Bigg],
\end{equation} 
where $\overline{\textbf{Ric}}$   denotes the Ricci curvature of the Riemannian metric $\alpha$.\\
It is well known that the Gaussian curvature of the Klein metric $\alpha_F=\frac{\sqrt{\left(r^2-|x|^2 \right) |\xi|^2+ \langle x, \xi \rangle ^2 }}{r^2-|x|^2}$ on $\mathbb{D}_K(1)$  is $-1$. Therefore, from \eqref{eqn2.1.12} the Riemann curvature of the Klein metric is given by 
 \begin{equation}\label{eqnn 4.3.710}
\overline{R^i_k}=- \left(  \delta^i_k \alpha_F^2-\alpha_F (\alpha_F)_k\xi^i \right),~ (\alpha_F)_k:=\frac{\partial \alpha_F}{\partial \xi ^k}.
\end{equation} 
Considering \eqref{eqnn 4.3.70} with \eqref{eqnn 4.3.710}, we get the desired expression for the Riemann curvature.\\
In \eqref{eqnn 4.3.710}, use $i=k$ and then sum over $i$ for  $i =1,2$, we have,
\begin{equation}\label{eqnn 4.3.711}
\overline{R^i_i}=-\alpha_F^2.
\end{equation} 
In view of \eqref{eqnn 4.3.521}, \eqref{eqnn 4.3.720} for the Funk metric $\mathcal{F}$,  given by \eqref{eqn 4.31}, the functions $\phi$ and $\psi$ can be explicitly calculated as follows:
 \begin{eqnarray}
\label{eqnn 4.3.74} \phi:=b_{i|j}\xi^i\xi^j=\frac{(1-r^2)\left[ (1-|x|^2)|\xi|^2+2\langle x , \xi \rangle^2\right] }{ (r^2-|x|^2)(1-|x|^2)^2},
\end{eqnarray}
and
\begin{eqnarray}
\nonumber \psi&=&b_{i|j|k}\xi^i\xi^j\xi^k=\left(\frac{\partial b_{i|j}}{\partial x^k}-b_{i|m}\bar{\Gamma}^m_{jk}-b_{j|m}\bar{\Gamma}^m_{ik}\right)\xi^i\xi^j\xi^k\\
\nonumber &=&\frac{2(1-r^2)\langle x , \xi \rangle}{(1-|x|^2)^3(r^2-|x|^2)^2}\Bigg[ (1-|x|^2)|\xi|^2(3r^2-2|x|^2-1)\\
\label{eqnn 4.3.75} &~&~~~~~~~~~~~~~~~~~~~~~~~~~~~~~~~~~~~~~~~~~~~~~-2 \langle x ,\xi \rangle^2(1+|x|^2-2r^2) \Bigg].
\end{eqnarray}
Using \eqref{eqnn 4.3.521} and \eqref{eqnn 4.3.720} in \eqref{eqnn 4.3.72}, we obtain
\begin{eqnarray}
\nonumber \label{eqnn 4.3.76} \tau_k=\frac{(1-r^2)(x^k |\xi|^2-\xi^k \langle x , \xi \rangle)}{\mathcal{F} (r^2-|x|^2)^2(1-|x|^2)}.
\end{eqnarray}
Clearly, 
\begin{equation}\label{eqnn 4.3.77}
\tau_k\xi^k=0.
\end{equation}
Consequently, substituting \eqref{eqnn 4.3.711},  \eqref{eqnn 4.3.77} in \eqref{eqnn 4.3.70},
 putting, $i=k$ and using over $i$ for  $i =1,2$, we obtain:
\begin{equation*}\label{eqnn 4.3.712}
 R^i_i=-\alpha_F^2+\Bigg[ 3\left( \frac{\phi}{2\mathcal{F}}\right)^2  -\frac{\psi}{2\mathcal{F}}\Bigg],
 \end{equation*}
 where $\phi$ , $\psi$ respectively, is given by \eqref{eqnn 4.3.74}, \eqref{eqnn 4.3.75}. \\
 Thus the Ricci curvature for  $\mathcal{F}$ is:
\begin{eqnarray}
 \nonumber \textbf{Ric}=R^i_i =\Bigg[3\left( \frac{\phi}{2\mathcal{F}}\right)^2  -\frac{\psi}{2\mathcal{F}} \Bigg]-\alpha_F^2.
\end{eqnarray}
Since for dimension $2$, $\textbf{K}=\frac{\textbf{Ric}}{\mathcal{F}^2}$,  we have the desired result.
 \end{proof}
\begin{corollary}
The flag curvature of the Funk-Finsler structure at center $\textbf{0}=(0,0)$ of the disc is negative constant i.e., independent of the direction. 
 \end{corollary}
\begin{proof}
    From \eqref{eqn 4.4.2A}, it is easy to see that, 
    \begin{equation*}
    \textbf{K}(\textbf{0},\xi)=-\left(1-\frac{3}{4}(1-r^2)^2\right)=-\left(1-\frac{3}{4}\left(\frac{4e^2}{e^2+1}\right)^2\right).
    \end{equation*}
\end{proof}
\begin{corollary}
The flag curvature of the Funk-Finsler structure on the Klein unit disc at any point $x$ on the circle centred at origin and radius $a<r$  along its tangential direction is a negative radial function. 
 \end{corollary}
\begin{proof}
   Since $\langle x(t) , \dot{x}(t) \rangle=0$. Therefore, from \eqref{eqn 4.4.2A}, we get
   \begin{equation*}
    \textbf{K}(x(t) , \dot{x}(t))=- \left(1-\frac{3}{4} \left(\frac{1-r^2}{1- a^2}\right)^2\right)=- \left(1-\frac{3}{4} \frac{1}{(1- a^2)^2}\left(\frac{4e^2}{e^2+1}\right)^2\right).
    \end{equation*}
\end{proof}

\subsection{Zermelo Navigation data for the Funk-Finsler structure $\mathcal{F}$}
It is well known that any Randers metric on a manifold $M$ has a Zermelo Navigation representation. For instant, if  $F=\alpha+\beta$ is given Randers metric with  $\alpha=\sqrt{a_{ij}(x)\xi^i\xi^j}$ and differential $1$-form $\beta=b_i(x)\xi^i$, satisfying $||\beta||^2_\alpha=a^{ij}b_ib_j < 1$. Then the Zermelo Navigation for this Randers metric is the triple $(M,h,W)$, where  $h=\sqrt{h_{ij}\xi^i \xi^j}$ with
\begin{equation*}
    h_{ij}= \epsilon (a_{ij}-b_ib_j), ~~ W^i=-\frac{b^i}{\epsilon},~ b^i=a^{ij}b_j~ \mbox{and}~ \epsilon=1-||\beta||^2_\alpha.
\end{equation*}
Moreover, $||W||_h=||\beta||_{\alpha}$.\\
%Then $(M, h)$ is a Riemannian manifold with Riemannian metric tensor $h=\sqrt{h_{ij}\xi^i \xi^j}$  and $W=(W^i)$ is a vector field on $M$ which satisfies $|W|^2_h=|\beta|^2_\alpha$.
%Thus the Zermelo data for 
%Randers metric is $(h, W)$. \\
Also given the Zermelo data, we can get back the Randers metric. 
And this 1-1 correspondence is useful in finding the geodesics of the Randers metric. % and other geometry in particular.
 See for more details  \cite{SSZ}, Example $1.4.3$.\\\\
In this subsection, we obtain the Zermelo data for the Funk-Finsler metric $\mathcal{F}$, which is a trivially a Randers metric.
We have,
 \begin{equation*}
    \mathcal{F}(x, \xi) = \frac{\sqrt{\left(r^2-|x|^2 \right) |\xi|^2+ \langle x , \xi \rangle ^2 }}{r^2-|x|^2}+\frac{(1-r^2) \langle x ,\xi \rangle}{(r^2-|x|^2)(1-|x|^2)}.
   \end{equation*}

We need to find  $h_{ij}, W^i$ defined above. From 
\eqref{eqn2.5.121}  we have,  

\begin{equation}
 \epsilon=  1- ||\beta_F||^2_{\alpha_F}=1-\frac{|x|^2(1-r^2)^2}{r^2(1-|x|^2)^2}=\frac{(r^2-|x|^2)(1-r^2|x|^2)}{r^2(1-|x|^2)^2}.  
\end{equation}

Employing \eqref{eqn2.5.118} and \eqref{eqn2.5.120}, we see

\begin{equation}\label{eqn 4.100A}
       h_{ij} =\frac{(1-r^2|x|^2)}{r^2(1-|x|^2)^4} \Bigg[ \delta_{ij} (1-|x|^2)^2+ x^ix^j ( 2-r^2 -|x|^2) \Bigg], 
 \end{equation} 

Clearly,
 \begin{equation}\label{eqn 4.101A}
   W=\left( W^i \right)=\left( \frac{(1-r^2)(1-|x|^2)}{(1-r^2|x|^2)}x^i  \right).
\end{equation}
Also, 
$$||W||^2_h=||\beta_F||^2_{\alpha_F}=\frac{|x|^2(1-r^2)^2}{r^2(r^2-|x|^2)^2}. $$
 Thus we have 
 \begin{proposition}
     The Zermelo Navigation data for the Funk-Finsler structure $\mathcal{F}$ on the Klein unit disc is given by $\left( \mathbb{D}_K(1), h, W\right)$ where the components of the Riemannian metric $h$ is given by \eqref{eqn 4.100A} and that of the vector field by \eqref{eqn 4.101A}.
 \end{proposition}
 \textbf{Funding} The First author is supported by CSIR-UGC Senior Research Fellowship, India with Reference No. 1076/(CSIR-UGC NET JUNE 2017).

\end{document}